\documentclass[12pt]{amsart}
\makeatletter
\@namedef{subjclassname@2020}{\textup{2020} Mathematics Subject Classification}
\makeatother
\usepackage{amssymb}
\usepackage{enumerate}

\theoremstyle{plain}
\newtheorem{theorem}{Theorem}[section]
\newtheorem{corollary}[theorem]{Corollary}
\newtheorem{lemma}[theorem]{Lemma}
\newtheorem{proposition}[theorem]{Proposition}
\newtheorem{conjecture}[theorem]{Conjecture}
\theoremstyle{definition}
\newtheorem{definition}[theorem]{Definition}
\newtheorem{remark}[theorem]{Remark}
\newtheorem{example}[theorem]{Example}

\newcommand{\eps}{\varepsilon}
\renewcommand{\phi}{\varphi}
\newcommand{\R}{\mathbb{R}}
\newcommand{\Z}{\mathbb{Z}}
\newcommand{\sq}{\mathbb{S}^{q-1}}
\newcommand{\hinf}{h_\infty}
\newcommand{\oo}{\mathcal{O}}
\newcommand{\vol}{\operatorname{Vol}}
\newcommand{\bcd}{\operatorname{BCD}}
\newcommand{\ahat}{\widehat{A}}

\newcommand{\bor}{controlled}
\newcommand{\cem}{coarse embedding}
\newcommand{\ceq}{coarse equivalence}
\newcommand{\clo}{close}
\newcommand{\cin}{coarse inverse}

\usepackage{graphicx}

\newcommand{\ee}{\hfill$\Diamond$\end{example}}

\begin{document}

\title{Coarse entropy}

\author{William Geller}
\address{Department of Mathematical Sciences, IUPUI, 402 N. Blackford
  Street, Indianapolis, IN 46202}
\email{wgeller@math.iupui.edu}
\author{Micha{\l} Misiurewicz}
\address{Department of Mathematical Sciences, IUPUI, 402 N. Blackford
  Street, Indianapolis, IN 46202}
\email{mmisiure@math.iupui.edu}

\thanks{This work was partially supported by grant number 426602
  from the Simons Foundation to Micha{\l} Misiurewicz.}

\date{April 10, 2020}



\begin{abstract}
Coarse geometry studies metric spaces on the large scale. 
Our goal here is to study dynamics from a coarse point of view. 
To this end we introduce a coarse version of topological entropy,
suitable for unbounded metric spaces, consistent with the coarse 
perspective on such spaces. As is the case with the usual topological 
entropy, the coarse entropy  measures the divergence of orbits. Following Bowen's ideas, 
we use $(n,\varepsilon)$-separated or $(n,\varepsilon)$-spanning sets.
However, we have to let $\varepsilon$ go to infinity rather than to
zero.
\end{abstract}

\maketitle

\section{Introduction}

For a continuous self map of a compact metric space, viewed as a 
discrete time dynamical system via iteration, the topological
entropy of Adler, Konheim, and McAndrew~\cite{AKM} can be seen
as a measure of the divergence of orbits. Rufus Bowen~\cite{B}  
extended this definition to the noncompact case, as did later
authors. 
While topological entropy is in some sense a global invariant,
a map may have large or even infinite topological entropy even 
if it acts as the identity on all but a small portion of the space.

Coarse (or large scale, or asymptotic) geometry, as developed by
Gromov~\cite{G} and many others in recent decades, considers
properties of metric spaces which, roughly speaking, are visible
to an observer at a vantage point receding to infinity. Since to
a coarse geometer all bounded metric spaces are equivalent to 
a point, the focus is on unbounded spaces, for example the Cayley
graph of a finitely generated infinite group. This example led
to the success of coarse geometry in geometric group theory,
where coarse properties of the Cayley graph (for instance the 
number of ends) give information about the group in question.
For more on coarse geometry, see e.g. Roe~\cite{R}.

Our goal here is to study dynamics from a coarse point of view.
To this end we introduce a coarse version of topological
entropy, suitable for unbounded metric spaces, consistent with
the coarse perspective on such spaces. This entropy should 
be invariant under a notion of coarse conjugacy of dynamical
systems, and so in particular should be insensitive to the
behavior of the map on a bounded invariant subset. This is
in stark contrast with the usual noncompact entropy.
The theory we develop will apply most usefully to \emph{controlled}
maps (see Section 2) of finite dimensional spaces, as we will see.

In Section 2 we introduce the notion of coarse conjugacy of maps
on metric spaces. In Section 3 we introduce the coarse entropy $\hinf$
of a map, show that it is an invariant of coarse conjugacy, and study
its behavior. In Section 4 we compute the coarse entropy of linear
maps on $\R^q$. We also compute $\hinf$ for certain homotheties and
relate it in this case to the box-counting dimension.
In Section 5 we provide examples showing what
can go wrong in infinite dimensional spaces. 

\section{Coarse conjugacy}

If we want to investigate coarse dynamics, we need to define coarse
conjugacy; this will play the same role as conjugacy in ordinary
dynamics. It turns out that this is not completely trivial.

Let us start by fixing terminology and notation. This is important,
since various authors use various terminology.

We will consider metric spaces, usually denoted $X,Y,Z$, with metric
that we will denote $d$ (in all spaces). Then we will consider a map
from the space to itself, and its iterations. To get the most general
results, we do not assume anything about the map. However, if we
restrict our attention to the class of \bor\ maps (see the definition
below), we get some additional properties. Note that not all
\bor\ maps are continuous and not all continuous maps are \bor.

We will call a map $\phi:X\to Y$ \emph{\bor}\footnote{Such maps are
also called \emph{bornologous}.}
if there is an increasing function $L:[0,\infty)\to[0,\infty)$
such that for every $x,x'\in X$
\[
d(\phi(x),\phi(x'))\le L(d(x,x')).
\]
If additionally for every $x,x'\in X$
\[
d(x,x')\le L(d(\phi(x),\phi(x'))),
\]
then $\phi$ is called a \emph{\cem}. If in addition to those two
inequalities there exists a constant $M>0$ such that for every $y\in
Y$ there exists $x\in X$ such that $d(y,\phi(x))\le M$, then $\phi$ is
called a \emph{\ceq}.

Clearly in the above definition we can replace $L$ by any increasing
function larger than or equal to $L$. Observe that for any increasing
function $L:[0,\infty)\to[0,\infty)$ there is a strictly increasing
continuous function $\hat{L}:[0,\infty)\to[0,\infty)$ such that
$L\le\hat{L}$ and $\lim_{t\to\infty}\hat{L}(t)=\infty$ (we leave the
proof of this simple fact to the reader as an entertainment).
Therefore in the future we will always assume that $L$ is strictly
inreasing, continuous and $\lim_{t\to\infty}L(t)=\infty$.

If for two maps $\phi,\phi':X\to Y$ there exists a constant $K>0$ such
that for every $x\in X$ we have $d(\phi(x),\phi'(x))\le K$, then we
will say that $\phi$ and $\psi$ are \emph{\clo}. Clearly, closeness is
an equivalence relation. A map $\psi:Y\to X$ is called a
\emph{\cin}\ of $\phi:X\to Y$ if $\psi\circ\phi$ is \clo\ to the
identity on $X$ and $\phi\circ\psi$ is \clo\ to the identity on $Y$.
The following facts are well-known.

\begin{lemma}\label{co0}
\begin{enumerate}[{\rm(a)}]
\item The composition of \bor\ maps (respectively, \cem s, \ceq s) is
  a \bor\ map (respectively, a \cem, a \ceq).
\item Every \ceq\ has a \cin, and this \cin\ is also a \ceq.
\item If a map is \clo\ to a \bor\ map (respectively, a \cem, a \ceq),
  then it is also a \bor\ map (respectively, a \cem, a \ceq).
\item If maps $\zeta,\zeta'$ are \clo\ and a map $\xi$ is such that
  the compositions $\zeta\circ\xi,\zeta'\circ\xi$ make sense, then
  those compositions are also \clo.
\item If maps $\zeta,\zeta'$ are \clo\ and a \bor\ map $\xi$ is
  such that the compositions $\xi\circ\zeta,\xi\circ\zeta'$ make
  sense, then those compositions are also \clo.
\end{enumerate}
\end{lemma}

\begin{remark}\label{co1}
In view of (c), when applying (a), each time before we apply the next
map in the composition, we can modify our map by a bounded amount.
\end{remark}

In the rest of this section we will be using the map $L$ and the
constant $K$ in the above sense. 
We will exploit the fact that $L$ can be replaced by a larger function
and $K$ by a larger constant to use the same $L$ and $K$ for several
maps under consideration.

The simplest idea for defining a coarse conjugacy between $f:X\to X$
and $g:Y\to Y$ would be to require that there exists a \ceq\
$\phi:X\to Y$ such that $\phi\circ f$ is \clo\ to $g\circ\phi$.
However, in Example~\ref{co6} we show that with this definition coarse
conjugacy would not be a symmetric relation. Therefore we need a
better definition.

\begin{definition}\label{co2}
Maps $f:X\to X$ and $g:Y\to Y$ are \emph{coarsely conjugate} if there
exists a \ceq\ $\phi:X\to Y$ with a \cin\ $\psi:Y\to X$ such that
$\phi\circ f$ is \clo\ to $g\circ\phi$ and $\psi\circ g$ is \clo\ to
$f\circ\psi$.
\end{definition}

\begin{conjecture}\label{co3}
If there exist \ceq s $\phi:X\to Y$ and $\psi:Y\to X$ such that
$\phi\circ f$ is \clo\ to $g\circ\phi$ and $\psi\circ g$ is \clo\ to
$f\circ\psi$ then $f$ is coarsely conjugate to $g$.
\end{conjecture}

\begin{example}\label{co4}
In Definition~\ref{co2} we cannot, in general, choose an arbitrary
\cin\ $\psi$ of $\phi$. For instance the maps $f,g:\R\to\R$,
given by $f(x)=x^2$ and $g(x)=x^2+2x$ are coarsely conjugate via
$\phi(x)=x-1$ and its \cin\ (in fact, inverse) $\psi(x)=x+1$,
but not via $\phi(x)=x-1$ and its \cin\ $\psi(x)=x$.
\ee

\begin{lemma}\label{co5}
Coarse conjugacy is an equivalence relation.
\end{lemma}

\begin{proof}
Clearly, coarse conjugacy is reflexive and symmetric. We will show
that it is also transitive.

Let $f:X\to X$, $g:Y\to Y$ and $h:Z\to Z$ be three maps; let $f$ be
coarsely conjugate to $g$ via $\phi$ and $\psi$, and let $g$ be
coarsely conjugate to $h$ via $\phi'$ and $\psi'$. We want to show
that $f$ is coarsely conjugate to $h$ by $\phi'\circ\phi$ and
$\psi\circ\psi'$.

By Lemma~\ref{co0}, $\phi'\circ\phi$ and $\psi\circ\psi'$ are \ceq s.
By Lemma~\ref{co0} and Remark~\ref{co1}, the maps
$\psi\circ\psi'\circ\phi'\circ\phi$ and
$\phi'\circ\phi\circ\psi\circ\psi'$ are \clo\ to the corresponding
identities, so $\psi\circ\psi'$ is a \cin\ of $\phi'\circ\phi$.

By the assumption and Lemma~\ref{co0}~(d), $h\circ\phi'\circ\phi$ is
\clo\ to $\phi'\circ h\circ\phi$. By the assumption and
Lemma~\ref{co0}~(e), $\phi'\circ h\circ\phi$ is \clo\ to
$\phi'\circ\phi\circ f$. Therefore $h\circ\phi'\circ\phi$ is \clo\ to
$\phi'\circ\phi\circ f$. Similarly, $f\circ\psi\circ\psi'$ is \clo\ to
$\psi\circ\psi'\circ h$.
\end{proof}

Now let us return to the question whether we really need $\psi$ in the
definition of coarse conjugacy.

\begin{example}\label{co6}
Let $X=\Z$ and $Y=\R$ and let $f:X\to X$ and $g:Y\to Y$ be defined by
the same formula $x\mapsto x^2$. If $\phi:X\to Y$ is the natural
embedding, $\phi(x)=x$, then clearly $\phi$ is a \ceq\ and $\phi\circ
f=g\circ\phi$. However, $f$ is not coarsely conjugate to $g$, because
there is no \ceq\ $\psi:Y\to X$ for which $\psi\circ g$ is \clo\ to
$f\circ\psi$.

Indeed, suppose that such $\psi$ exists. Then for every $x>0$ the set
$\psi([x,x+1])$ is contained in an interval of length $L(1)$, so it has
at most $L(1)+1$ elements. Therefore the set $f(\psi([x,x+1]))$ has
also at most $L(1)+1$ elements, so the set $\psi(g([x,x+1]))$ has at
most $(2K+1)(L(1)+1)$ elements. However, the interval $g([x,x+1])$ has
length $2x+1$, so there must be a point $n\in\Z$ whose preimage under
$\psi$ has diameter at least $\frac{2x+1}{(2K+1)(L(1)+1)}$. This means
that there are $y,z\in\R$ with $\psi(y)=\psi(z)$ and
$d(y,z)\ge\frac{2x+1}{(2K+1)(L(1)+1)}$. Therefore,
\[
\frac{2x+1}{(2K+1)(L(1)+1)}\le L(0),
\]
which is clearly not true if $x$ is sufficiently large.
\ee

However, if we focus on controlled maps, we can dispense with $\psi$
in the definition of coarse conjugacy.  In fact, we have
\begin{proposition}\label{co7}
Consider maps $f:X\to X$ and $g:Y\to Y$ for which there exists a
\ceq\ $\phi:X\to Y$ such that $\phi\circ f$ is \clo\ to $g\circ\phi$
and $g$ is \bor. Then $f$ is also \bor\ and for any \cin\ $\psi$ of
$\phi$ the maps $f$ and $g$ are coarsely conjugate via $\phi$ and
$\psi$.
\end{proposition}

\begin{proof}
Let $\psi:Y\to X$ be a \cin\ of $\phi$. In the proof we will be using
all the time Lemma~\ref{co0} (and once Remark~\ref{co1}). The map
$\psi\circ g\circ\phi$ is \bor. Since $\psi\circ\phi$ is \clo\ to the
identity and $\phi\circ f$ is \clo\ to $g\circ\phi$, we see that $f$
is \clo\ to $\psi\circ\phi\circ f$, which is \clo\ to $\psi\circ
g\circ\phi$. Thus, $f$ is \bor.

Further, we see that $f\circ\psi$ is \clo\ to $\psi\circ
g\circ\phi\circ\psi$, which is \clo\ to $\psi\circ g$. Thus,
$f\circ\psi$ is \clo\ to $\psi\circ g$, so $f$ and $g$ are coarsely
conjugate via $\phi$ and $\psi$.
\end{proof}

The assumption that $g$ is \bor\ is important. The following example
shows that it cannot even be replaced by the assumption that $f$ is
\bor.

\begin{example}\label{co8}
Let $X=\R$, $Y=\{(x,y)\in\R^2: y\in\{0,1\}\}$, $f:X\to X$ be the
identity, and $g(x,0)=(x,0)$, $g(x,1)=(x^2,1)$. If $\phi:X\to Y$ is
given by $\phi(x)=(x,0)$, then it is a \ceq\ and $\phi\circ
f=g\circ\phi$. However, there is no \ceq\ $\psi:Y\to X$ such that
$\psi\circ g$ is \clo\ to $f\circ\psi$. Indeed, if such $\psi$ exists,
then by restricting it to $\R\times\{1\}$ and identifying this line
with $\R$, we see that we can use Proposition~\ref{co7} to deduce that
the map $x\mapsto x^2$ is \bor. Since it is clearly not \bor, such
$\psi$ cannot exist.
\ee

A coarse conjugacy between maps need not work for their iterates 
if the maps are not controlled.

\begin{example}\label{co9}
Take $X=Y=[2,\infty)$, $f(x)=x^2$, $g(x)=x^2+\frac1x$, and both $\phi$
and $\psi$ equal to the identity. clearly, the pair $(\phi,\psi)$ is
a coarse conjugacy between $f$ and $g$. However, it is not a coarse
conjugacy between $f^2$ and $g^2$. Indeed,
\[
g^2(x)-f^2(x)=-2x+\frac1{x^2}-\frac{x}{x^3-1}
\]
is not bounded.

However, it is easy to check that $f^2$ and $g^2$ are coarsely
conjugate via $\phi'(x)=x-\frac1{2x^2}$ and $\psi'(x)=x+\frac1{2x^2}$.
\ee

\begin{conjecture}\label{co10}
If $f$ and $g$ are coarsely conjugate then so are $f^n$ and $g^n$ for
all natural $n$.
\end{conjecture}

\begin{lemma}\label{co11}
Consider maps $f,f':X\to X$, $g,g':Y\to Y$ and a \ceq\ $\phi:X\to Y$
such that $g,g'$ are \bor, $\phi\circ f$ is \clo\ to $g\circ\phi$ and
$\phi\circ f'$ is \clo\ to $g'\circ\phi$. Then $\phi\circ f'\circ f$
is \clo\ to $g'\circ g\circ\phi$.
\end{lemma}

\begin{proof}
By Proposition~\ref{co7}, $f$ and $f'$ are \bor. If $\psi$ is a
\cin\ of $\phi$, then, as in the proof of Proposition~\ref{co7}, $f$
is \clo\ to $\psi\circ g\circ\phi$. Now, by Lemma~\ref{co0},
$\phi\circ f'\circ f$ is \clo\ to $\phi\circ f'\circ\psi\circ
g\circ\phi$, which is \clo\ to $g'\circ\phi\circ\psi\circ g\circ\phi$,
which is \clo\ to $g'\circ g\circ\phi$.
\end{proof}

{}From Proposition~\ref{co7} and Lemma~\ref{co11} we get immediately
the following corollary.

\begin{corollary}\label{co12}
If $f$ and $g$ are coarsely conjugate and $g$ is \bor, then for any
natural $n$ the maps $f^n$ and $g^n$ are coarsely conjugate via the
same \ceq s as $f$ and $g$.
\end{corollary}

\section{Coarse entropy}

Let $(X,d)$ be a metric space, and $f:X\to X$ a map. We want to define
coarse entropy of $f$ using similar ideas as in the usual definitions
of topological entropy. Mimicking the original definition of Adler,
Konheim and McAndrew~\cite{AKM} can be difficult, since the space is
not compact and the map is not necessarily continuous. Thus, we will
try to mimic the definition of Bowen~\cite{B}. However, we have to
incorporate the idea of closeness replacing equality. This means
that instead of orbits we should use $\delta$-pseudoorbits.
Fortunately, we know that $\delta$-pseudoorbits work well with 
Bowen's definition (see~\cite{M}). Of course, we have to replace
$\eps$ going to 0 by $R$ going to infinity.

Thus, we define the \emph{coarse entropy} of $f$ as
\[
\hinf(f)=\lim_{\delta\to\infty} \lim_{R\to\infty}
\limsup_{n\to\infty} \frac1n \log s(f,n,R,\delta,x_0),
\]
where $s(f,n,R,\delta,x_0)$ is the supremum of cardinalities of
$R$-separated sets of $\delta$-pseudo\-orbits of $f$ of length $n$
starting at $x_0$. As usual, a $\delta$-pseudoorbit of $f$ of length $n$
starting at $x_0$ is a sequence $(x_0,x_1,\dots,x_n)$ such that
$d(f(x_i),x_{i+1})\le\delta$ for $i=0,1,\dots,n-1$. The distance
between the pseudoorbits $(x_0,x_1,\dots,x_n)$ and
$(y_0,y_1,\dots,y_n)$ is the maximum of the distances $d(x_i,y_i)$ over
$i=0,1,\dots,n$. A set is $R$-separated if the distance between each
two distinct elements of this set is at least $R$.

The value of $\hinf(f)$ in the above definition does not depend on the
choice of $x_0\in X$. Indeed, if $y_0\in X$ is another point, then a
$\delta$-pseudoorbit starting at $x_0$ can have $y_0$ as the next
element (and vice versa).

Given $f:X\to X$ and $g:Y\to Y$, we will say that the map $f$ is
\emph{coarsely embedded} in the map $g$ if there exists a \cem\ of spaces
$\phi:X\to Y$ such that $\phi\circ f$ is \clo\ to $g\circ\phi$.

\begin{theorem}\label{t1}
If $f$ is coarsely embedded in $g$ then $\hinf(f)\le\hinf(g)$.
\end{theorem}

\begin{proof}
We will keep the same notation as in the definitions. Suppose that
$(x_0,x_1,\dots,x_m)$ is a $\delta$-pseudoorbit of $f$ in $X$. For
$i=0,1,\dots,m-1$ we have
\[
d\big(g(\phi(x_i)),\phi(x_{i+1})\big)\le d\big(g(\phi(x_i)),
\phi(f(x_i))\big)+ d\big(\phi(f(x_i)),\phi(x_{i+1})\big)\le
K+L(\delta).
\]
Thus, the image under $\phi$ of a $\delta$-pseudoorbit in $X$ is a
$(L(\delta)+K)$-pseudoorbit of $g$ in $Y$. On the other hand, if two
$\delta$-pseudoorbits in $X$ are $R$-separated, then their images in
$Y$ are $L^{-1}(R)$-separated. Therefore,
\[
\limsup_{n\to\infty} \frac1n \log s(g,n,L^{-1}(R),L(\delta)+K,
\phi(x_0))\ge \limsup_{n\to\infty} \frac1n \log s(f,n,R,\delta,x_0).
\]

The quantities $R$ and $L^{-1}(R)$ go to infinity simultaneously.
Similarly, $\delta$ and $L(\delta)+K$ go to infinity simultaneously. In
such a way we obtain $\hinf(f)\le\hinf(g)$.
\end{proof}

\begin{corollary}\label{c1}
If $f$ is coarsely embedded in $g$ and $g$ is coarsely embedded in $f$ then
\break $\hinf(f)=\hinf(g)$. Therefore, the coarse entropy is an
invariant of coarse conjugacy. In particular, if we change the metric
$d$ to a metric that is bi-Lipschitz equivalent, or quasi-isometric,
to $d$, the coarse entropy will not change.
\end{corollary}

\begin{remark}\label{r1}
Maps $f$ and $g$ may each coarsely embed in the other without being
coarsely conjugate. Let $X$ be the binary tree with edges of unit
length and $Y$ be $X$ with a ray attached at the root, each with the
path metric (see Figure~\ref{fig1}). If $f$ and $g$ are the identity
maps on $X$ and $Y$ respectively, then $g$ coarsely embeds in $f$ as in
the first diagram and $f$ coarsely embeds in $g$ via the inclusion. But
$f$ and $g$ are not coarsely conjugate since $X$ and $Y$ are not
coarsely equivalent, as their boundaries (a Cantor set, and the union of 
a Cantor set and an isolated point, respectively) are not homeomorphic.
\end{remark}

\begin{figure}
\begin{center}
\includegraphics[width=120truemm]{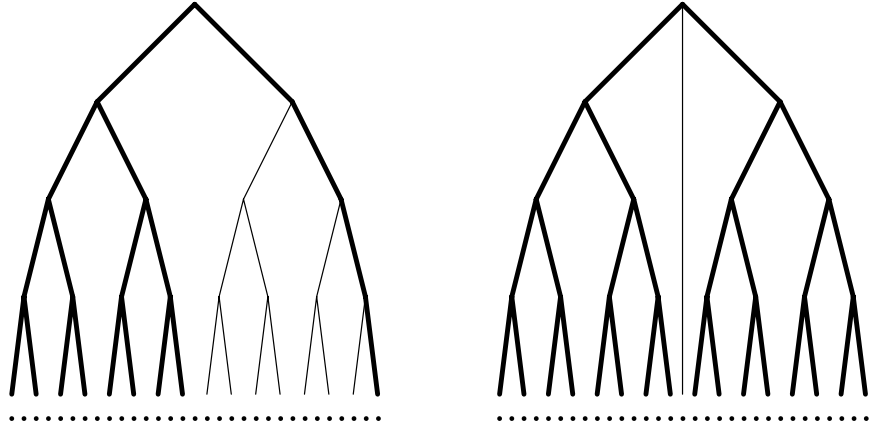}
\caption{Mutual coarse embedding without coarse conjugacy.}\label{fig1}
\end{center}
\end{figure}

\begin{example}\label{e1}
This is an example where $f$ and $g$ are homeomorphisms, they are
conjugate via a Lipschitz (but not bi-Lipschitz) homeomorphism $\phi$
(that is, $\phi\circ f=g\circ\phi$), but $\hinf(g)>\hinf(f)$.

Let $X=Y$ be the half-plane $\{(x,y)\in\R^2:y\ge 0\}$. Let $f:X\to X$
be given by the formula $f(x,y)=(2x,y)$. The identity coarsely embeds $f$
into the linear map of $\R^2$ to itself given by the same formula,
and, as we will see later, the coarse entropy of that map is $\log 2$.
Therefore, by Theorem~\ref{t1}, we have $\hinf(f)\le\log 2$.

The map $\phi:X\to Y$ maps each horizontal line $H_t=\{(x,y)\in\R^2:
y=t\}$ to itself by squeezing linearly the segment (in the variable
$x$) $[-e^t,e^t]$ to the segment $[-1,1]$ and translating the
remaining two half-lines. Thus, if $-e^y\le x\le e^y$, then $\phi(x,y)
=(xe^{-y},y)$; if $x>e^y$ then $\phi(x,y)=(x-e^y+1,y)$; and if $x<-e^y$
then $\phi(x,y)=(x+e^y-1,y)$. Clearly, $\phi$ is a homeomorphism.

Finally, we set $g=\phi\circ f\circ\phi^{-1}$. Let us estimate the
coarse entropy of $g$. Take $x_0=(0,0)$. If $\delta>1$ then for every
$x\in[-1,1]$ there is a $\delta$-pseudoorbit
\[
((0,0),(0,\delta),(0,2\delta),\dots,(0,(n-2)\delta),(x,(n-2)\delta))
\]
of length $n-1$. Therefore, there is a $\delta$-pseudoorbit of length
$n$ starting at $x_0$ and ending at $z=g(x,(n-2)\delta)$. The point
$z$ can be any chosen point of the image under $g$ of
$[-1,1]\times\{(n-2)\delta\}$. To find this image, we use the
definition of $g$. Its second coordinate is $t:=(n-2)\delta$. For the
first coordinate, we take the interval $[-e^t,e^t]$, multiply by 2 to
get $[-2e^t,2e^t]$, and shorten by $e^t-1$ from both sides, to get
$[-e^t-1,e^t+1]$. It follows that we can find an $R$-separated set of
$\delta$-pseudoorbits of $g$ of length $n$ starting at $x_0$, which
has cardinality $(2e^t+2)/R-1$, so
\[
s(g,n,R,\delta,x_0)\ge \frac2R e^{(n-2)\delta}-1.
\]
Thus,
\[
\limsup_{n\to\infty} \frac1n \log s(g,n,R,\delta,x_0)\ge\delta,
\]
and hence, $\hinf(g)=\infty$.
\ee

A subset of a metric space is \emph{$R$-spanning} if for every element
of the space there is an element of the subset at distance less than
$R$. Let $r(f,n,R,\delta,x_0)$ be the infimum of cardinalities of
$R$-spanning sets of $\delta$-pseudoorbits of $f$ of length $n$
starting at $x_0$.

\begin{theorem}\label{t2}
We have
\[
\hinf(f)=\lim_{\delta\to\infty} \lim_{R\to\infty}
\limsup_{n\to\infty} \frac1n \log r(f,n,R,\delta,x_0).
\]
\end{theorem}

\begin{proof}
Any maximal $R$-separated set is also $R$-spanning. This proves that
\[
r(f,n,R,\delta,x_0)\le s(f,n,R,\delta,x_0),
\]
so
\[
\hinf(f)\ge\lim_{\delta\to\infty} \lim_{R\to\infty}
\limsup_{n\to\infty} \frac1n \log r(f,n,R,\delta,x_0).
\]

On the other hand, in each ball of radius $R$ centered in an element
of an $R$-spanning set there may be at most one element of a
$2R$-separated set. This proves that
\[
r(f,n,R,\delta,x_0)\ge s(f,n,2R,\delta,x_0),
\]
so
\[
\hinf(f)\le\lim_{\delta\to\infty} \lim_{R\to\infty}
\limsup_{n\to\infty} \frac1n \log r(f,n,R,\delta,x_0).
\]
\end{proof}

\begin{theorem}\label{t3}
For any $k\ge 1$ we have $\hinf(f^k)\le k\hinf(f)$. If additionally
$f$ is \bor, then $\hinf(f^k)=k\hinf(f)$.
\end{theorem}

\begin{proof}
Clearly, we have
\[
s(f,kn,R,\delta,x_0)\ge s(f^k,n,R,\delta,x_0).
\]
Therefore
\[\begin{split}
k\cdot\limsup_{m\to\infty} \frac1m \log s(f,m,R,\delta,x_0)&\ge
k\cdot\limsup_{n\to\infty} \frac1{kn} \log s(f,kn,R,\delta,x_0)\\
&\ge\limsup_{n\to\infty} \frac1n \log s(f^k,n,R,\delta,x_0),
\end{split}\]
and thus, $\hinf(f^k)\le k\hinf(f)$.

Assume now that $f$ is \bor\ with function $L$. If
$(x_0,x_1,\dots,x_k)$ is a $\delta$-pseudoorbit of $f$, then by
induction on $k$ we get $d(f^k(x_0),x_k)\le \eta_k$, where
\[
\eta_k=\delta+L(\delta)+L^2(\delta)+\dots+L^{k-1}(\delta).
\]
Thus, if $(x_0,x_1,\dots,x_{nk})$ is a $\delta$-pseudoorbit of $f$,
then $(x_0,x_k,\dots,x_{nk})$ is an $\eta_k$-pseudo\-orbit of $f^k$.

Moreover, if $(x_0,x_1,\dots,x_i)$ and $(y_0,y_1,\dots,y_i)$ are
$\delta$-pseudoorbits of $f$, then we get
\[
d(x_i,y_i)\le d(f^i(x_0),f^i(y_0))+d(f^i(x_0),x_i)+d(f^i(y_0),y_i)
\le L^i(d(x_0,y_0))+2\eta_i
\]
for every $i>0$. We may assume that $L(t)\ge t$ for every $t$, and
then, if $i\le k$,
\[
d(x_i,y_i)\le  L^k(d(x_0,y_0))+2\eta_k.
\]
Therefore, if $d(x_i,y_i)\ge L^k(0)+2\eta_k$, then
\[
d(x_0,y_0)\ge L^{-k}(d(x_i,y_i)-2\eta_k).
\]
Changing indices, we see that if $(x_0,x_1,\dots,x_{nk})$ and
$(y_0,y_1,\dots,y_{nk})$ are $\delta$-pseudo\-orbits of $f$, then for
$m=1,2,\dots,n-1$ and $i=1,2,\dots,k$, if $d(x_{mk+i},y_{mk+i})\ge
L^k(0)+2\eta_k$, then
\[
d(x_{mk},y_{mk})\ge L^{-k}(d(x_{mk+i},y_{mk+i})-2\eta_k).
\]

If the distance between those two pseudoorbits is at least $R$, and
$R\ge L^k(0)+2\eta_k$, then there are $m$ and $i$ such that
$d(x_{mk+i},y_{mk+i})\ge R$, and hence the distance between the
$\eta_k$-pseudoorbits $(x_0,x_k,\dots,x_{nk})$ and
$(y_0,y_k,\dots,y_{nk})$ of $f^k$ is at least
\[
S_k=L^{-k}(R-2\eta_k).
\]
This proves that
\[
s(f,kn,R,\delta,x_0)\le s(f^k,n,S_k,\eta_k,x_0).
\]

If $j=kn+i$ with $i\le k$, then $s(f,j,R,\delta,x_0)\le
s(f,k(n+1),R,\delta,x_0)$. Therefore
\[\begin{split}
&\limsup_{j\to\infty} \frac1j \log s(f,j,R,\delta,x_0)\le
\limsup_{j\to\infty}\frac1j\log s(f,k\lceil j/k\rceil,R,\delta,x_0)\\
=&\limsup_{n\to\infty} \frac1{kn} \log s(f,kn, R,\delta,x_0)\le
\frac1k\limsup_{n\to\infty} \frac1n\log s(f^k,n,S_k,\eta_k,x_0).
\end{split}\]
With $\delta$ (and therefore also $\eta_k$) fixed, $R$ and $S_k$ go to
infinity simultaneously, so
\[
k\cdot\lim_{R\to\infty}\limsup_{j\to\infty} \frac1j \log
s(f,j,R,\delta,x_0)\le \lim_{S\to\infty}\limsup_{n\to\infty}
\frac1n\log s(f^k,n,S,\eta_k,x_0).
\]
Now, $\delta$ and $\eta_k$ go to infinity simultaneously, and
thus, $\hinf(f^k)\ge k\hinf(f)$.
\end{proof}

\begin{example}\label{e2}
This example shows that in the above theorem, if we do not assume $f$
is \bor, then it can happen that $\hinf(f^k)<k\hinf(f)$.

Let $X$ be a disjoint union of rectangles $P_n$, $n=0,1,2,\dots$.
Rectangle $P_{2m}$ has size $1\times 2^m$ and rectangle $P_{2m+1}$ has
size $2^m\times 1$. Let $c_n$ be the center of the rectangle $P_n$. On
each rectangle the metric is the maximum of horizontal and vertical
distances. If $x\in P_n$ and $y\in P_m$ for $n<m$, then
\[
d(x,y)=d(x,c_n)+d(y,c_m)+(n+1)+(n+2)+\dots+m
\]
(that is, the distance between $P_n$ and $P_{n+1}$ is $n+1$).

The map $f$ maps $P_n$ onto $P_{n+1}$ by a linear map that preserves
the horizontal and vertical directions. Thus, as we apply $f$
repeatedly, the rectangles get alternately stretched horizontally while
contracting vertically, and stretched vertically while contracting
horizontally. However, $f^2$ only stretches each rectangle in one
direction by the factor $2$.

Assume that $\delta>2$. For $m>0$ we construct some special
$\delta$-pseudoorbits of the length $2m+2$. We set $x_0=c_0$ and as
$x_1$ we can take any point of $P_0$. For the next $2m+1$ steps we
just follow the orbit of $x_1$. If we choose locations of $x_1$ at the
vertices of a square grid with vertical and rectangular distances of
size $R/2^m$, then for two distinct points of this set the distance
between the last or the last but one elements of the corresponding
pseudoorbits will be at least $R$. There are $4^m/R$ of those
vertices, so $s(f,2m+2,R,\delta,x_0)\ge 4^m/R$. Therefore,
\[
\limsup_{n\to\infty} \frac1n \log s(f,n,R,\delta,x_0)\ge\log 2
\]
so $\hinf(f)\ge\log 2$.

On the other hand, when we look at $\delta$-pseudoorbits for $f^2$,
then once we get into $P_n$ with $n>\delta$, we have to move in each
step from $P_i$ to $P_{i+2}$. This means that up to a multiplicative
and an additive constant, the maximal cardinality of an
$R$-separated set of $\delta$-pseudoorbits will not be larger
than for multiplication by $2$ on the real line. We will
see later that the coarse entropy of this multiplication is $\log 2$,
and thus $\hinf(f^2)\le\log 2<2\log 2\le2\hinf(f)$.
\ee

\begin{theorem}\label{t4}
Let $f:X\to X$ and $g:Y\to Y$ be maps. Then
\[
\hinf(f\times g)\le\hinf(f)+\hinf(g).
\]
\end{theorem}

\begin{proof}
In $X\times Y$ we can take the max metric. If $(x_0,x_1,\dots,x_n)$ is
a $\delta$-pseudoorbit in $X$, and $(y_0,y_1,\dots,y_n)$ is a
$\delta$-pseudoorbit in $Y$, then $((x_0,y_0),(x_1,y_1),\dots,
(x_n,y_n))$ is a $\delta$-pseudoorbit in $X\times Y$. Therefore, if
$E_X$ is an $R$-spanning set of $\delta$-pseudoorbits of $f$ of length
$n$ starting at $x_0$ and $E_Y$ is an $R$-spanning set of
$\delta$-pseudoorbits of $g$ of length $n$ starting at $y_0$, then
$E_X\times E_Y$ (understood in an obvious sense) is an $R$-spanning
set of $\delta$-pseudoorbits of $f\times g$ of length $n$ starting at
$(x_0,y_0)$. Hence,
\begin{equation}\label{eq1}
r(f\times g,n,R,\delta,(x_0,y_0))\le r(f,n,R,\delta,x_0)\cdot
r(g,n,R,\delta,y_0).
\end{equation}
Therefore,
\[\begin{split}
&\limsup_{n\to\infty}\frac1n\log r(f\times g,n,R,\delta,(x_0,y_0))\\
\le&\limsup_{n\to\infty}\left(\frac1n\log r(f,n,R,\delta,x_0)+
\frac1n\log r(g,n,R,\delta,y_0)\right)\\
\le&\limsup_{n\to\infty}\frac1n\log r(f,n,R,\delta,x_0)+
\limsup_{n\to\infty}\frac1n\log r(g,n,R,\delta,y_0).
\end{split}\]
By Theorem~\ref{t2}, we get $\hinf(f\times g)\le\hinf(f)+\hinf(g)$.
\end{proof}

\begin{example}\label{e3}
This example shows that even if we assume that if $f$ and $g$ increase
distances at most 2 times and do not decrease distances, we may not
get equality in Theorem~\ref{t4}.

We define the spaces $X$ and $Y$ in a similar way as in
Example~\ref{e2}, except that instead of rectangles, we take segments
of the real line. The point $c_n$ will be the left endpoint of the
$n$th segment, and the distance in the space is defined in a similar
way as in Example~\ref{e2}. The length of the zeroth segment is 1. The
lengths of the next segments will be determined by the maps $f$ and
$g$. Both of them map the $n$th segment onto the $(n+1)$st one in the
linear way; it will be the multiplication by 1 or 2. If $2^{k^2}\le
n<2^{(k+1)^2}$, then if $k$ is even then $f$ multiplies by 1 and $g$
by 2; if $k$ is odd then $f$ multiplies by 2 and $g$ by 1.

We may assume that $\delta>3$. Then, if $n=2^{k^2}$ with $k$ odd, the
length of the $n$th segment in $X$ is at least
$2^{2^{k^2}-2^{(k-1)^2}}$, so
\[
\frac1n\log s(f,n,R,\delta,x_0)\ge\frac1n\log\frac
{2^{2^{k^2}-2^{(k-1)^2}}}R=\frac{2^{k^2}-2^{(k-1)^2}}{2^{k^2}}\log 2
-\frac1n\log R,
\]
and therefore $\hinf(f)\ge\log 2$. Similarly, $\hinf(g)\ge\log 2$.

To obtain an upper estimate on $\hinf(f\times g)$, in view of
the inequality~\eqref{eq1}, it is enough to
construct for each $n$ $R$-spanning sets of $\delta$-pseudoorbits
starting at $x_0$ for $f$, and starting at $y_0$ for $g$, with
relatively small cardinalities. Let us do it for $f$. Denote the
$n$th segment $P_n$ and its length $l_n$.

Concentrate first on the $\delta$-pseudoorbits $(x_0,x_1,\dots,x_n)$
for which $x_i\in P_i$ (and $x_0=c_0$). We may assume that $\delta$ is
large and $R$ is much larger. Set $m=\lfloor\frac R{6\delta}
\rfloor$ and $n=km$ for some integer $k>0$. Partition each interval
$P_{jm}$ into subsegments of length $R/3$ (one of them may be
shorter). Since $f$ does not shorten the distances between points, if
our $\delta$-pseudoorbit has $jm$th point in a given subsegment of the
partition of $P_{jm}$, then its $(j-1)m$th point is in some specific
segment of length not larger than $R/3+2m\delta$ of $P_{(j-1)m}$.
Since $R/3+2m\delta\le 2R/3$, this segment can intersect at most 3
elements of the partition of $P_{(j-1)m}$. Thus, if we code our
$\delta$-pseudoorbits by the elements of the partitions through which
they pass, the number of the valid codes will be not larger than
$3^k(3l_n/R+1)$. On the other hand, if two $\delta$-pseudoorbits have
the same code, then their distance is at most $2R/3<R$. Thus, there
exists an $R$-spanning set of $\delta$-pseudoorbits of length $n$ of
cardinality at most $3^k(3l_n/R+1)$.

Now we have to deal with the fact that there are $\delta$-pseudoorbits
for which not necessarily $x_i\in P_i$ for each $i$. Once a
$\delta$-pseudoorbit gets to a segment $P_i$ with $i>\delta$, it has
to move to the next segment with each application of $f$. On the other
hand, if $R$ is large enough, if we distinguish between two points
only if their distance is at least $R$, the union of the segments
$P_i$ with $i\le\delta$ is seen as one point. Therefore our estimate
of the cardinality of an $R$-spanning set has to be only multiplied by
$n$.

The other thing we have to deal with is that we obtained our estimate
only for $n$s which are multiples of $m$. However, when taking a limit
with respect to $n$, it does not matter whether we divide by $n$, or
by $n-m$, or by anything in between.

If we use the estimate we obtained for $f$ and the analogous estimate
for $g$ (where the length of the $n$th segment is $q_n$), we get
\[\begin{split}
r(f\times g,n,R,\delta,(x_0,y_0))&\le 3^k(3l_n/R+1)n\cdot3^k(3q_n/R+1)
n\\ &\le 3^{n\cdot 12\delta/R}(3l_n/R+1)(3q_n/R+1)n^2.
\end{split}\]
Taking into account that $l_nq_n=2^n$, we get
\[
\limsup_{n\to\infty} \frac1n \log r(f\times g,n,R,\delta,(x_0,y_0))\le
12\frac\delta{R}\log 3+\log 2.
\]
To compute the coarse entropy, we go to infinity with $R$ before we go
to infinity with $\delta$, so $\hinf(f\times g)\le\log 2$.

Thus, in our example
\[
\hinf(f\times g)\le\log 2<2\log 2=\hinf(f)+\hinf(g).
\]
\ee

The idea of the above example is that since in the definition of
coarse entropy we take the upper limit, for distinct maps those upper
limits can be limits along different subsequences. However, if the
maps are equal, we can take the same subsequences. Therefore we have
the following result (suggested to us by Mariusz Lema\'nczyk).

\begin{proposition}\label{lem}
Let $f:X\to X$ be a map and $k\ge 2$ an integer. Then
\[
\hinf(F)=k\hinf(f),
\]
where $F=f\times f\times\dots\times f$ (k times).
\end{proposition}

\begin{proof}
In $X^k$ we take the max metric. If $(x_0^i,x_1^i,\dots,x_n^i)$ are
$\delta$-pseudoorbits of $f$ in $X$ for $i=1,2\dots,k$, then
$((x_0^1,x_0^2, \dots,x_0^k),\dots,(x_n^1,x_n^2,\dots,x_n^k))$ is a
$\delta$-pseudoorbit of $F$ in $X^k$. Therefore, if $E$ is an
$R$-separated set of $\delta$-pseudoorbits of $f$ of length $n$
starting at $x_0$, then $E^k$ is an $R$-separated set of
$\delta$-pseudoorbits of $F$ of length $n$ starting at
$(x_0,x_0,\dots,x_0)$. Hence,
\[
s(F,n,R,\delta,(x_0,x_0,\dots,x_0))\ge(s(f,n,R,\delta,x_0))^k,
\]
and thus $\hinf(F)\ge k\hinf(f)$. Together with Theorem~\ref{t4}
applied inductively, we get $\hinf(F)=k\hinf(f)$.
\end{proof}

\section{Linear maps}

One of the basic tests whether our definition is good is whether the
entropy of a linear map of a finite dimensional euclidean space is
correct, that is, whether it is the sum of positive logarithms of the
absolute values of eigenvalues. We will start with the expanding case.

\begin{lemma}\label{l1}
If $f:\R^q\to\R^q$ is a linear map with all eigenvalues of absolute
value larger than $1$ and the absolute value of the determinant of $f$
is $\Lambda$, then $\hinf(f)=\log\Lambda$.
\end{lemma}

\begin{proof}
By changing the basis in $\R^q$ we may assume that for the Euclidean
norm $\|\cdot\|$ there exists $\lambda>1$ such that for every
$x\in\R^q$ we have $\|f(x)\|\ge\lambda\|x\|$.

Let $x_0$ be the origin of $\R^q$. Fix $\delta>0$ and consider the set
$\oo_n$ of all $\delta$-pseudoorbits of $f$ of length $n$ starting at
$x_0$. For such a $\delta$-pseudoorbit $(x_0,x_1, \dots,x_n)$ we will
call $x_n$ its \emph{final term}. Let $K_n$ be the set of final terms
of all elements of $\oo_n$. In particular, $K_1=B(\delta)$, where
$B(t)$ denotes the closed ball centered at $x_0$ with radius $t$.
Therefore, $K_n\supset f^{n-1}(B(\delta))$. It follows that if $E$ is
an $R$-spanning set in $\oo_n$, then the set of final terms of $E$ has
to $R$-span $f^{n-1}(B(\delta))$. Thus, if $\vol$ denotes the
$q$-dimensional volume, then the cardinality $|E|$ of $E$
satisfies
\[
|E|\ge\frac{\vol(f^{n-1}(B(\delta)))}{\vol(B(R))}=\Lambda^{n-1}
\frac{\vol(B(\delta))}{\vol(B(R))}=\Lambda^{n-1}\left(\frac\delta
R\right)^q,
\]
so $\hinf(f)\ge\log\Lambda$.

We claim that the $\delta$-pseudoorbits of $f$ have the following
\emph{shadowing property}: if $(x_0,x_1,\dots,x_n)$ is a
$\delta$-pseudoorbit, then the orbit $(f^{-n}(x_n),f^{-n+1}(x_n),
\dots, x_n)$ is $\frac\delta{\lambda-1}$-close to it
(remember that by our assumptions $f$ is a bijection). Indeed, by
induction we get
\[
\|x_k-f^{-n+k}(x_n)\|\le\frac\delta\lambda+\frac\delta{\lambda^2}
+\dots+\frac\delta{\lambda^{n-k}}<\frac\delta{\lambda-1}.
\]

In particular, we get $\|f^{-n}(x_n)\|<\frac\delta{\lambda-1}$, so
$K_n\subset f^n\left(B(\frac\delta{\lambda-1})\right)$. The set
$f^n\left(B(\frac\delta{\lambda-1})\right)$ is a $q$-dimensional
ellipsoid of volume $\Lambda^n\vol\left(B\left(\frac\delta
{\lambda-1}\right)\right)$. This ellipsoid is contained in a
$q$-dimensional box $A_n$ of volume $C_1\Lambda^n$, where the constant
$C_1$ does not depend on $n$. The thickness of $A_n$ (the minimal
length of its edges) is at least $C_2\lambda^n$ for some constant
$C_2>0$ independent of $n$. Therefore for a constant $S>0$ the set
$A_n$ can be covered by $C_3\Lambda^n$ subsets of diameter smaller
than $S$, where $C_3$ does not depend on $n$. Consequently, we can
find sets $E_n$ which are $S$-dense in $K_n$ and $|E_n|\le
C_3\Lambda^n$.

If $(x_0,x_1,\dots,x_n)$ and $(y_0,y_1,\dots,y_n)$ are
$\delta$-pseudoorbits and $\|x_n-y_n\|\le S$, then the distances
between $(x_0,x_1,\dots,x_n)$ and $(f^{-n}(x_n),f^{-n+1}(x_n),\dots,
x_n)$, and between $(y_0,y_1,\dots,y_n)$ and $(f^{-n}(y_n),
f^{-n+1}(y_n),\dots,y_n)$, are smaller than $\frac\delta{\lambda-1}$,
while the distance between $(f^{-n}(x_n),f^{-n+1}(x_n),\dots,x_n)$ and
$(f^{-n}(y_n),f^{-n+1}(y_n),\dots,y_n)$ is not larger than $S$. Thus,
the distance between $(x_0,x_1,\dots,x_n)$ and $(y_0,y_1, \dots,y_n)$
is smaller than $\frac{2\delta}{\lambda-1}+S$. Therefore, if for each
element $x\in E_n$ we choose one $\delta$-pseudoorbit from $\oo_n$
whose final term is $x$, we get an $R$-spanning subset in $\oo_n$ of
cardinality not larger than $C_3\Lambda^n$, where $R=\frac{2\delta}
{\lambda-1}+S$. This gives us the inequality $\hinf(f)\le\log\Lambda$.
\end{proof}

\begin{lemma}\label{l2}
If $f:\R^q\to\R^q$ is a Lipschitz continuous map with Lipschitz
constant $\lambda>1$ then $\hinf(f)\le q\log\lambda$.
\end{lemma}

\begin{proof}
Fix $\delta>0$ and set $S=\frac{2\delta}{\lambda-1}$. Then fix a large
integer $m$ and $R>2S\lambda^m$. If $(x_0,x_1,\dots,x_m)$ and
$(y_0,y_1,\dots,y_m)$ are $\delta$-pseudoorbits then by induction we
see that for $i=0,1,\dots,m$
\[
\|x_i-y_i\|\le \lambda^i\|x_0-y_0\|+2\delta\frac{\lambda^i-1}
{\lambda-1} <\lambda^m(\|x_0-y_0\|+S),
\]
so in particular, if $\|x_0-y_0\|\le S$, then $\|x_i-y_i\|\le
2S\lambda^m$.

There is a constant $C$ such for every $\alpha>\beta>0$ every subset
of $\R^q$ of diameter less than $\alpha$ can be partitioned into less
than $C(\alpha/\beta)^q$ subsets of diameter less than $\beta$. Using
this, we can define by induction for each $j=1,2,\dots,k$ a family
$A_j$ of sets of diameter less than $S$, such that for every
$\delta$-pseudoorbit $(x_0,x_1,\dots,x_{km})$ of $f$, where $x_0$ is
the origin, $x_{jm}$ belongs to exactly one element $B\in A_j$, and
then there are less than $C(2S\lambda^m/S)^q=C2^q\lambda^{mq}$
elements of $A_{j+1}$ to which $x_{(j+1)m}$ can belong. Specifying the
elements of $A_1,A_2,\dots, A_k$ to which the corresponding terms of
our $\delta$-pseudoorbit belong, gives us a set of
$\delta$-pseudoorbits of diameter less than $2S\lambda^m<R$. The
number of such sets is at most $(C2^q\lambda^{mq})^k$. Thus,
\[
\frac1{km}\log r(f,km,R,\delta,x_0)\le\frac1{km}\log
(C2^q\lambda^{mq})^k= \frac1m\log(C2^q\lambda^{mq})=\frac1m\log(C2^q)
+q\log\lambda.
\]
By the same argument as in Example~\ref{e3}, we can replace $km$ with
any $n$ and pass to the limit with $n$, obtaining
\[
\limsup_{n\to\infty}\frac1n\log r(f,n,R,\delta,x_0)\le
\frac1m\log(C2^q) +q\log\lambda.
\]

As we take the limit of the left-hand side of the above inequality as
$R\to\infty$, we can assume that $m\to\infty$, since the only
condition for $m$ is that $R>2S\lambda^m$. After taking the last
limit, as $\delta\to\infty$, we get $\hinf(f)\le q\log\lambda$.
\end{proof}


\begin{theorem}\label{t5}
If $f:\R^q\to\R^q$ is a linear map, then $\hinf(f)=\log\Lambda$, where
$\Lambda$ is the absolute value of the product of all eigenvalues of
$f$ that have absolute value larger than $1$.
\end{theorem}

\begin{proof}
By changing the metric in $\R^q$, we may consider $\R^q$ as the
product of two Euclidean spaces: $X$ corresponding to the eigenvalues
of $f$ with absolute values larger than 1, and $Y$ corresponding to
the eigenvalues of $f$ with absolute values less than or equal to 1.
In this model, $f=g\times h$, where $g:X\to X$ is a linear map with
all eigenvalues of absolute value larger than $1$ and the absolute
value of the determinant equal to $\Lambda$, and $h:Y\to Y$ is a
linear map with all eigenvalues of absolute value smaller than or
equal to $1$.

By Lemma~\ref{l1}, $\hinf(g)=\log\Lambda$. To find the coarse entropy
of $h$, note that for every $\eps>0$ we can further change the metric
in $Y$ in such a way that $h$ is Lipschitz continuous with Lipschitz
constant $1+\eps$. Then, by Lemma~\ref{l2}, we get $\hinf(h)\le
q\log(1+\eps)$. Since $\eps$ is arbitrary, we get $\hinf(h)=0$.

Now, by Theorem~\ref{t4} we get $\hinf(f)\le\log\Lambda$, and by
Theorem~\ref{t1} we get $\hinf(f)\ge\log\Lambda$. Thus,
$\hinf(f)=\log\Lambda$.
\end{proof}

Let us consider another interesting example, where we can express the
coarse entropy in terms of the properties of the map and the phase
space. Let us recall the notion of the \emph{box-counting dimension}
(or rather ball-counting dimension, but in our case it will be the
same) of a bounded space $X$. It is equal to
\[
\bcd(X)=\lim_{\eps\to0}\frac{\log r(X,\eps)}{-\log\eps}
\]
(if the limit exists), where $r(X,\eps)$ is the minimum
cardinality of any $\eps$-spanning subset of $X$.

\begin{example}\label{e6}
Let $\sq$ be the unit sphere in $\R^q$. Let $A\subset\sq$
be a set having box-counting dimension. Set
\[
X=\{tx\in\R^q:t\ge 0,\ x\in A\}.
\]
Take $\lambda>1$ and define $f:X\to X$ by $f(x)=\lambda x$. We will
show that
\[
\hinf(f)=(\bcd(A)+1)\log\lambda.
\]

Set $\ahat=\{tx\in\R^q:0\le t\le 1,\ x\in A\}$. We will start by
showing that $\bcd(\ahat)=\bcd(A)+1$.

Let $E$ be an $\eps$-spanning set in $A$ and $D$ an $\eps$-spanning
set in $[0,1]$. Then $\{tx:t\in D,\ x\in E\}$ is a $2\eps$-spanning
set in $\ahat$. Therefore
\begin{equation}\label{eq2}
\limsup_{\eps\to0}\frac{\log r(\ahat,2\eps)}{-\log\eps}\le
\bcd(A)+\bcd([0,1])=\bcd(A)+1.
\end{equation}

For $t\in[1/2,1]$, let $F_t$ be the projection to the sphere
$\mathbb{S}_t$ of radius $t$ centered at the origin:
$F_t(y)=t\frac{y}{\|y\|}$. Set $E_t=\{y\in E:t-\eps<\|y\|<t+\eps$. If
$\eps$ is sufficiently small and $t\in[1/2,1]$, then whenever
$\big|\|y\|-\|x\|\big|<\eps$, $\|x\|=t$, and $\|y-x\|<\eps$, then
$\|F_t(y)-x\|<2\eps$. Thus, $|E|\ge r(\mathbb{S}_t\cap A,2\eps)
=r(A,2\eps/t)\ge r(A,4\eps)$. Dividing $[1/2,1]$ into $m$ intervals of
length larger than $2\eps$ and considering as $t$ the centers of those
intervals, we see that $|E|\ge mr(A,4\eps)$. We can
take $m>1/(3\eps)$, so
\[
\liminf_{\eps\to0}\frac{\log r(\ahat,\eps)}{-\log\eps}\ge
\lim_{\eps\to0}\frac{\log r(A,4\eps)-log(3\eps)}{-\log\eps}=\bcd(A)+1.
\]
Together with~\eqref{eq2}, we get $\bcd(\ahat)=\bcd(A)+1$.

Now we have to prove that $\hinf(f)=\bcd(\ahat)$. We will use the same
method as in the proof of Lemma~\ref{l1} and we will use terminology
and some results from this proof.

If $E$ is an $R$-spanning set in $\oo_n$, then the set of final terms
of $E$ has to $R$-span $f^{n-1}(B(\delta)\cap X)=B(\lambda^{n-1}\delta)
\cap X$. However, covering $B(\lambda^{n-1}\delta)\cap X$ with balls
of radius $R$ is the same as covering $\ahat$ with balls of radius
$R/(\lambda^{n-1}\delta)$. Thus,
\[
r(f,n,R,\delta,x_0)\ge r\left(\ahat,\frac{R}{\lambda^{n-1}\delta}
\right).
\]
We have
\[\begin{split}
\lim_{n\to\infty}\frac1n\log r\left(\ahat,\frac{R}{\lambda^{n-1}\delta}
\right)=\lim_{n\to\infty}\frac{\log r\left(\ahat,\frac{R}{\lambda^{n-1}
\delta}\right)}{-\log \frac{R}{\lambda^{n-1}\delta}}&\left(\frac1n
\log\frac\delta{R}+\frac{n-1}n\log\lambda\right)\\
&=\bcd(\ahat)\log\lambda.
\end{split}\]
Therefore, $\hinf(f)\ge\bcd(\ahat)\log\lambda$.

To get the opposite inequality, we use the fact that $K_n\subset
f^n\left(B(\frac\delta{\lambda-1})\cap X\right)=B\left(\frac{\delta
\lambda^n}{\lambda-1}\right)\cap X$. Covering $B\left(\frac{\delta
\lambda^n}{\lambda-1}\right)\cap X$ with balls of radius $S$ is the
same as covering $\ahat$ with balls of radius $S(\lambda-1)/
(\delta\lambda^n)$. Taking $S=R-\frac{2\delta}{\lambda-1}$, we get an
$R$-spanning subset in $\oo_n$ of cardinality not larger than
\[
r\left(\ahat,\frac{S(\lambda-1)}{\delta\lambda^n}\right)=
r\left(\ahat,\frac{R(\lambda-1)-2\delta}{\delta\lambda^n}\right).
\]
Hence,
\[\begin{split}
\limsup_{n\to\infty}\frac1n\log r(f,n,R,\delta,x_0)&\le
\limsup_{n\to\infty}\frac1n\log r\left(\ahat,\frac{R(\lambda-1)
-2\delta}{\delta\lambda^n}\right)\\&=\bcd(\ahat)\lim_{n\to\infty}
\frac1n\log\frac{\delta\lambda^n}{R(\lambda-1)-2\delta}=
\bcd(\ahat)\log\lambda.
\end{split}\]
Therefore, $\hinf(f)\le\bcd(\ahat)\log\lambda$, so $\hinf(f)=
\bcd(\ahat)\log\lambda$.
\ee

\section{Entropy of the identity map}

It seems unavoidable that whatever reasonable definition of the coarse
entropy we try, if the space is large enough, then the entropy of the
identity is positive (or even infinite). Here ``large enough''
basically means that the dimension is infinite.

\begin{example}\label{e4}
Let $X$ be the space $l_\infty$ of bounded real sequences, with the
sup norm, and let $f:X\to X$ be the identity map. Fix $\delta,R>0$.
Let $x_0$ be the zero sequence. If $n\ge R/\delta$ then for every $k$
there exists a $\delta$-pseudoorbit of length $n$ starting at $x_0$
and ending at the sequence whose only non-zero term is the $k$th one,
and it is equal to $R$. The set of those $\delta$-pseudoorbits is an
$R$-separated set of cardinality infinity. This proves that
$\hinf(f)=\infty$.
\ee

The above example and easily constructed similar ones are based on the
property of the space $X$ that for every $R$ there are bounded sets
with $R$-separated infinite subsets. However, there is an example of a
space where the closure of every bounded set is compact, so every
$R$-separated subset of a bounded set is finite, but nevertheless the
identity has infinite coarse entropy.

\begin{example}\label{e5}
Let $X$ be the half-line $[0,\infty)$ with the space $\R^{2^k}$
attached at every integer $k$ (with the origin on our half-line). The
metric in $X$ is ``along the space'', so for example if $x\in\R^k$ and
$y\in\R^l$ with $k\ne l$, then $d(x,y)=\|x\|+|l-k|+\|y\|$. Let $f$ be
the identity on $X$.

Fix $\delta$ and $R$, and let $x_0$ be the point $0$ on our half-line.
If $n$ is large, look at the $\delta$-pseudoorbits from $x_0$ that
first go with step $\delta$ along the half-line, and when they reach
$k=k(n)=\lfloor \delta n-2R\rfloor$, they start spreading out in
$\R^{2^k}$. Their final distance from $x_0$ is $\delta n$ or anything
less, so their final distance from the origin in $\R^{2^k}$ is
approximately $2R$ or anything less. Thus, among the final points on
those pseudoorbits are in particular all points in $\R^{2^k}$ of the
form $(0,\dots,0,R,0,\dots,0)$. They form an $R$-separated set and
there are $2^k$ of them. This means that $s(n,\delta,R,x_0)\ge
2^{k(n)}$. We get
\[
\limsup_{n\to\infty} \frac1n \log s(n,\delta,R,x_0)\ge
\limsup_{n\to\infty} \frac1n (\delta n-2R)\log 2=\delta\log 2.
\]
Therefore, $h_\infty(f)=\infty$.
\ee


\begin{thebibliography}{AKM}

\bibitem[AKM]{AKM}
R. L. Adler, A. G. Konheim and M. H. McAndrew,
\emph{Topological entropy},
Trans. Amer. Math. Soc. \textbf{114} (1965), 309-319.

\bibitem[B]{B}
R. Bowen,
\emph{Entropy for group endomorphisms and homogeneous spaces},
Trans. Amer. Math. Soc. \textbf{153} (1971), 401-414; erratum:
Trans. Amer. Math. Soc. \textbf{181} (1973), 509-510.

\bibitem[G]{G}
M. Gromov,
\emph{Asymptotic invariants of infinite groups},
in ``Geometric group theory, Vol. 2'', London 1993. 

\bibitem[M]{M}
M. Misiurewicz,
\emph{Remark on the definition of topological entropy},
in ``Dynamical Systems and Partial Differential Equations'', Caracas
1986, pp. 65-68.

\bibitem[R]{R}
J. Roe,
\emph{Lectures on Coarse Geometry},
AMS University Lecture Series \textbf{31}, Providence, 2003. 

\end{thebibliography}
\end{document}